\documentclass[10pt,twoside]{amsart}
\usepackage{amssymb,amsmath,amsfonts,amsthm,graphicx,amscd}
\usepackage{enumerate,mathrsfs,longtable,bm,float,}
\usepackage{amsmath,amsfonts,amssymb,amsthm,bm,float}
\usepackage{enumerate}
\usepackage{amscd}
\usepackage[all,cmtip,line]{xy}
\DeclareMathAlphabet{\mathpzc}{OT1}{pzc}{m}{it}

\numberwithin{equation}{section}

\theoremstyle{plain}

\newtheorem{lem}{Lemma}[section]
\newtheorem{lemma}[lem]{Lemma}

\newtheorem{thm}[lem]{Theorem}
\newtheorem{prop}[lem]{Proposition}
\newtheorem{cor}[lem]{Corollary}
\theoremstyle{definition}
\newtheorem{exa}[lem]{Example}
\newtheorem{rem}[lem]{Remark}

\newtheorem{defn}[lem]{Definition}
\newtheorem{Conj}[lem]{Conjecture}

\begin{document}

\baselineskip 16truept

\title{Diameter and girth of zero divisor graph of multiplicative lattices }

\author{Vinayak Joshi and Sachin Sarode}
\address{\rm Department of Mathematics, University of Pune,
Pune-411007, India.} \email{vvj@math.unipune.ac.in \\
sarodemaths@gmail.com}
\thanks{ This research is supported by Board of College and University Development, University of Pune}
\subjclass[2010]{Primary 05C15, Secondary
06A12}

\date{October 18, 2013} \maketitle



\begin{abstract}
In this paper, we study the zero divisor graph $\Gamma^m(L)$ of a multiplicative lattice $L$. We prove under certain conditions that for a  reduced multiplicative lattice $L$ having more than two minimal prime elements, $\Gamma^m(L)$ contains a cycle and $gr(\Gamma^{m}(L)) = 3$. This essentially proves that for a reduced ring $R$ with more than two minimal primes,  $gr(\mathbb{AG}(R)))= 3$ which settles the conjecture of Behboodi and Rakeei \cite{br2}. Further, we have characterized the diameter of $\Gamma^m(L)$.

\end{abstract}
\noindent{\bf Keywords:} Zero-divisor graph, reduced multiplicative
lattice, minimal prime element.

\section{Introduction}
In recent years, lot of attention have been given to the study of
zero divisor graphs of algebraic structures and ordered
structures. The idea of a zero divisor graph of a commutative ring
with unity was introduced by Beck \cite{B}. He was particularly
interested in the coloring of commutative rings with unity. Many
mathematicians like Anderson and Naseer \cite{AN}, Anderson and
Livingston \cite{AL},  F. DeMeyer, T. McKenzie and K. Schneider
\cite{DMS}, Maimani, Pournaki and Yassemi \cite{mpy}, Redmond
\cite{red} and Samei \cite{sam} investigated the interplay between
properties of the algebraic structure and graph theoretic
properties.

The zero divisor graphs of ordered structures are well studied by
Alizadeh et. al. \cite{ali1,ali2} Hala\v{s} and Jukl \cite{hj},
Hala\v{s} and L\"anger \cite{hl}, Joshi \cite{j}, Joshi et.al.
\cite{ja, jaa, jk, jwp, jwp1}, Nimbhorkar et.al \cite{nwl} etc.

In ring theory, the structure of a ring $R$ is closely related to
ideal's behavior more than elements. Hence Behboodi and Rakeei
\cite{br1, br2} introduced the concept of annihilating ideal-graph
$\mathbb{AG}(R)$ of a commutative ring $R$  with unity, where the
vertex set $V(\mathbb{AG}(R))$ is the set of non-zero ideals with
non-zero annihilator, that is, for a non-zero  ideal $I$ of  $R$,
$I\in V(\mathbb{AG}(R))$ if and only if there exists a non-zero
ideal $J$ of $R$ such that $I  J=(0)$ and two distinct vertices
$I$ and $J$ are adjacent if and only if $IJ=(0)$ and studied the
properties of rings and its annihilating ideal-graphs. In
\cite{br2}, Behboodi and Rakeei raised the following conjecture.

\begin{Conj} \textit{Let $R$  be a reduced ring with
more than two minimal primes. Then $gr(\mathbb{AG}(R)))= 3$.}
\end{Conj}

It is interesting to observe that the set $Id(R)$ of all ideals of
a commutative ring $R$ with unity forms a modular, compactly
generated, 1-compact  multiplicative lattice in which product of
two compact element is compact (see Definition \ref{1.1}) and the
annihilating ideal-graph of a commutative ring $R$ with unity is
nothing but the zero divisor graph of the multiplicative lattice
of all ideals of $R$ where the vertex set is the set of non-zero
zero divisors and vertices $a$ and $b$ are adjacent if and only if
$ab=0$. Hence to study the annihilating ideal-graphs of
commutative ring with unity, a multiplicative lattice becomes a
tool. This motivate us to define and study the zero divisor graph
of a multiplicative lattice. It is natural  to ask the following
question and the affirmative answer to this question solves
Conjecture 1.1. of Behboodi and Rakeei \cite{br2}.

\noindent{\bf Question 1:} {\it Let  $L$ be a reduced, $1$-compact, compactly generated lattice
with product of two compact element is compact. Let $\Gamma^m(L)$ be  a zero divisor graph of a multiplicative lattice $L$. Is $gr(\Gamma^{m}(L)) = 3$, if $L$ has more than two minimal prime elements?}\\

In this paper, we study the zero divisor graph $\Gamma^m(L)$ of a multiplicative lattice $L$. We prove under certain conditions that for a  reduced multiplicative lattice $L$ having more than two minimal prime elements, $\Gamma^m(L)$ contains a cycle and $gr(\Gamma^{m}(L)) = 3$. This essentially proves that for a reduced ring $R$ with more than two minimal primes,  $gr(\mathbb{AG}(R)))= 3$ which settles the conjecture of Behboodi and Rakeei \cite{br2}. Further, we have characterized the diameter of $\Gamma^m(L)$.

   Now, we begin with necessary concepts and terminology.

\begin{defn} \label{1.1}  A non-empty subset $I$ of a lattice $L$ is said to be \textit{semi-ideal}, if $x \leq a\in I$ implies that $x\in I$. A semi-ideal $I$ of $L$ is said to be an \textit{ideal},
if for $a,b \in I$, $a \vee b \in I$. A proper ideal (semi-ideal)
$I$ of a lattice $L$ is said to be \textit{prime} if $a\wedge b
\in I$ implies $a \in I $ or $b \in I$. Dually, we have concept of
a prime filter (semi-filter). A prime ideal (semi-ideal)[element]
$I$  is a \textit{minimal prime ideal (semi-ideal)[element]} if
there is no prime ideal (semi-ideal)[element] $Q$ such that $
\{0\} \subsetneqq Q \subsetneqq I$. A filter is said to be
\textit{maximal} if it is a maximal element of the poset of
filters.

For $a\in L$, the set $(a]=\{x \in L~|~ x \leq a\}$ is called the
\textit{principal ideal generated by a}. Dually, we have a concept
of a principal filter $[a)$ generated by $a$.

A lattice $L$ is said to be \textit{complete}, if for any subset
$S$ of $L$, we have $\bigvee S,\bigwedge S\in L$.

A complete lattice $L$ is said to be a \textit{multiplicative
lattice}, if there is defined a binary operation $``\cdot"$ called
multiplication on $L$ satisfying the following conditions:
\begin{enumerate} \item $a \cdot b=b \cdot a$, for all $a,b\in L$; \item
$a \cdot(b \cdot c)=(a \cdot b)\cdot c$, for all $a,b,c\in L$;
\item $a \cdot(\vee_{\alpha} b_{\alpha})=\vee_{\alpha}(a \cdot
b_{\alpha})$, for all $a,b_{\alpha}\in L$; \item $a \cdot b \leq a
\wedge b$ for all $a,b\in L$; \item $a \cdot 1=a$, for all $a\in
L$.
\end{enumerate}

An element $ c $ of a complete lattice $ L $ is said to be
\textit{compact}, if $ c \leq \bigvee_{\alpha} a_{\alpha} $
implies that $ c \leq \bigvee _{i=1}^{n} a_{\alpha_{i}} $, where $
n \in \mathbb{Z}^{+} $. The set of all compact elements of a
lattice $ L $ is denoted by $ L_{*} $. A lattice $L$ is said to be
\textit{compactly generated} or \textit{algebraic}, if for every $
x \in L $, there exist $ x_{\alpha} \in L_{*} $, $\alpha \in
\Lambda$ such that $ x = \vee_{\alpha} x_{\alpha} $, that is,
every element is a join of compact elements.

   A multiplicative lattice $L$ is said to be \textit{$1$-compact }, if $1$ is a compact element of $L$.
A multiplicative lattice $L$ is said to be \textit{compact }, if
every element of $L$ is a compact element.

 An element $p \neq 1$ of a multiplicative lattice $L$ is said to be \textit{prime}, if $a
\cdot b \leq p$ implies either $a \leq p$ or $b \leq p$.
Equivalently, an element $p \neq 1$ of a $1$-compact, compactly
generated lattice $L$ is said to be \textit{prime} if $a \cdot b
\leq p$ for $a, b \in L_{*}$ implies either $a \leq p$ or $b \leq
p$.

A non-empty subset $S$ of $L_{*}$ in $1$-compact, compactly
generated lattice is said to be \textit{multiplicatively closed },
if $s_{1},~s_{2} \in S$, then $s_{1} \cdot s_{2} \in S$.

As $L$ is a complete lattice,  it follows that $L$ admits
residuals: for every pair $a, ~b \in L$, there exists an element
$(a:b) = \bigvee \{x \;|\;  x\cdot b \leq a\} \in L$ such that for
any $x \in L$ , $x \cdot b \leq a \Leftrightarrow x \leq (a:b)$.
Clearly, $a \leq (a:b)$ for all $a,~ b \in L$.


 In a multiplicative lattice $L$, an element $a \in L$  is
said to be \textit{nilpotent}, if $a^{n}= 0$, for some $ n \in
\mathbb{Z}^{+}$ and  $L$ is said to be \textit{reduced}, if the
only nilpotent element is 0. The set of all nilpotent elements of $L$ is denoted by $Nil(L)$.

Let $a$ be an element of a multiplicative lattice. We define
$a^*=\bigvee\{x \in L \;|\; a^n \cdot x=0\}$. If $L$ is
reduced, then $a^*=\bigvee\{x \in L \;|\; a \cdot x=0\}$.

 A lattice $L$ with 0 is said to be \textit{0-distributive} if $a
\wedge b=0=a\wedge c$ then $a\wedge (b \vee c)=0$; see Varlet
\cite{var}. The concept of 0-distributive poset can be found in
\cite{jm, jw}.

An ideal $I$ of a lattice $L$ is said to be \textit{semiprime}, if $a\wedge b, a\wedge c \in I$ imply that $a\wedge (b \vee c) \in I$. Note that $I=\{0\}$ is a semiprime ideal of $L$ if and only if $L$ is 0-distributive.
\end{defn}
 Let  $G$ be a graph and $x, y$ be distinct vertices in $G$. We denote by $d(x,y)$  the length of a
shortest path from $x$ to  $y$, if it exists and put
$d(x,y)=\infty$ if no such path exists. The {\it diameter } of $G$, denoted $diam(G)$,
is zero if $G$ is the graph on one vertex and is
$diam(G)=sup\{d(x,y) \;|x$ and $y$ are distinct vertices of $G$\}
otherwise. A {\it cycle } in a graph $G$ is a path that begins and
ends at the same vertex. The {\it girth } of $G$, written $gr(G)$,
is the length of the shortest cycle in $G$ (and $gr(G)=\infty$ if
$G$ has no cycles).  The chromatic number of $G$ is denoted by $\chi(G)$.  Thus, $\chi(G)$ is the minimum number of colors which can be assigned to the elements of $L$ such that adjacent elements receive different colors. If this number is not finite, write $\chi(G) = \infty$.
A subset $C$ of $G$ is a clique if any two distinct vertices of $C$ are adjacent.  If $G$ contains a clique with $n$ elements and every
clique has at most $n$ elements then the clique number of $G$ is
$Clique(G) = n$. If the sizes of the cliques are not bounded, then
$Clique(G) = \infty$.  We always have $\chi(G) \geqslant Clique(G)$.


For undefined concepts in lattices and graphs, see Gr\"atzer
\cite{Gg} and Harary \cite{H} respectively.

\section{Zero-divisor graph of a multiplicative lattice}

Joshi \cite{j} introduced the zero-divisor graph of a poset with
respect to an ideal $I$. We mentioned this definition, when a poset is a lattice.

\begin{defn}\label{2.1.} Let $I$ be an ideal of a  lattice $L$. We associate an undirected and simple graph,
called the {\it zero-divisor graph of $L$ with respect to $I$},
denoted by $\Gamma_I(L)$ in which the set of vertices is
$\{x\not\in I \;|\; x \wedge y \in I$ for some $y\not\in I\}$ and
two distinct vertices $a,b$ are adjacent if and only if $a \wedge
b \in I$. When $I=\{0\}$, then it is simply denoted by $\Gamma(L)$
and in  this case the above definition of Lu and Wu \cite{lw}.

\end{defn}

We illustrate this concept with an example.

\begin{exa}\label{2.2} The lattice  $L$ and its zero divisor graph $\Gamma(L)$ (in the sense of
Joshi \cite{j}) is shown in Figure 1.

\begin{figure}[h]
\unitlength .61mm 
\linethickness{0.4pt}
\ifx\plotpoint\undefined\newsavebox{\plotpoint}\fi 
\begin{picture}(108.266,80.581)(15,10)
\put(33.5,56.581){\circle{3.5}} \put(42.54,71.192){\circle{3.5}}
\put(49.372,47.804){\circle{3.5}} \put(43.25,83.581){\circle{3.5}}
\put(33.5,39.331){\circle{3.5}} \put(42.25,30.081){\circle{3.5}}
\put(42.25,73.831){\line(0,-1){.25}}
\put(41,70.331){\line(-2,-3){8}}
\put(42.75,81.831){\line(0,-1){8.5}}
\put(42.75,73.581){\line(0,-1){.5}}
\put(55.25,48.081){\makebox(0,0)[rc]{c}}
\put(42,24.331){\makebox(0,0)[cb]{0}}
\put(37.5,72.331){\makebox(0,0)[lc]{$d$}}
\put(42.75,90.581){\makebox(0,0)[ct]{1}}
\multiput(43.379,31.469)(.047347107,.121181818){121}{\line(0,1){.121181818}}
\put(33.393,54.751){\line(0,-1){13.612}}
\put(33.393,41.192){\line(0,-1){.053}}
\multiput(34.234,37.775)(.048395683,-.047258993){139}{\line(1,0){.048395683}}
\multiput(43.379,69.729)(.0473,-.168175){120}{\line(0,-1){.168175}}
\put(71.812,31.416){\circle{3.784}}
\put(83.164,48.444){\circle{3.784}}
\put(65.4,48.497){\circle{3.784}}
\multiput(65.19,46.658)(.047211864,-.114915254){118}{\line(0,-1){.114915254}}
\multiput(82.218,46.815)(-.04748731,-.070162437){197}{\line(0,-1){.070162437}}
\put(87.947,48.55){\makebox(0,0)[rc]{$b$}}
\put(71.444,26.581){\makebox(0,0)[cb]{$c$}}
\put(60.879,48.55){\makebox(0,0)[lc]{$a$}}
\put(41.444,17.74){\makebox(0,0)[cb]{$L$}}
\put(39.444,7.5){\makebox(0,0)[cb]{$(a)$}}
\put(71.444,17.74){\makebox(0,0)[cb]{$\Gamma(L)$}}
\put(72.444,7.5){\makebox(0,0)[cb]{$(b)$}}
\put(121.516,28.749){\circle{3.5}} \put(99.189,48){\circle{3.354}}
\put(121.439,48){\circle{3.354}}
\put(99.689,28.75){\circle{3.354}}
\put(119.939,47.25){\line(0,1){.25}}
\put(98.939,23.5){\makebox(0,0)[cb]{$a$}}
\put(120.939,23.5){\makebox(0,0)[cb]{$b$}}
\put(121.439,54.25){\makebox(0,0)[ct]{$c$}}
\put(98.689,53){\makebox(0,0)[ct]{$d$}}
\multiput(100.387,46.838)(.0556396648,-.0474189944){358}{\line(1,0){.0556396648}}
\put(100.755,48.467){\line(1,0){19.025}}
\put(119.465,48.467){\line(1,0){.368}}
\put(99.126,46.365){\line(0,-1){15.977}}
\put(99.126,30.44){\line(0,-1){.105}}
\put(121.199,30.44){\line(0,1){15.977}}
\multiput(120.148,46.995)(-.0515582656,-.0474281843){369}{\line(-1,0){.0515582656}}
\put(101.386,28.443){\line(1,0){18.5}}
\put(101.491,28.391){\line(-1,0){.158}}
\put(108.939,17.5){\makebox(0,0)[cb]{$\Gamma^m(L)$}}
\put(110.939,7.5){\makebox(0,0)[cb]{$(c)$}}
\put(70.939,.5){\makebox(0,0)[cb]{\bf Figure 1}}
\put(29.5,39.25){\makebox(0,0)[cc]{$a$}}
\put(29.75,57.75){\makebox(0,0)[cc]{$b$}}
\end{picture}
\end{figure}

\end{exa}

The following result is essentially proved by D. D. Anderson \cite{AD} for $r$-lattices  guarantees the existence of a sufficient
supply of prime elements in $L$.

\begin{thm}[{D.D. Anderson \cite[Theorem 2.2]{AD}}]\label{1.3.}
Let $L$ be a $1$-compact, compactly generated lattice with $L_{*}$
as a multiplicatively closed set. Suppose $a \in L$ and $ t \nleq
a$ for all $t\in S$, where $S$ be a multiplicatively closed subset
of $L$. Then there is a prime element $p$ of $L$ such that $a \leq
p$ and maximal with respect to $t \nleq p$ for all $t \in S$.
\end{thm}

A proof of the following corollary follows from the fact that
there exist a compact element $(0 \neq) c \leq a$, as $L$ is
compactly generated and set $S = \{c^{n}\}$ where $ n \in
\mathbb{Z}^{+}$ is a multiplicatively closed set.

\begin{cor}\label{1.4.} Let $L$ be a reduced $1$-compact compactly generated lattice with $L_{*}$ as a
multiplicatively closed set and $a \neq 0 \in L$. Then there is a
prime element $p$ not containing $a$. Moreover, every prime
element contains a minimal prime element.
\end{cor}

Now, we introduced the zero-divisor graph $\Gamma^m(L)$ of a
multiplicative lattice $L$ and illustrate with an example.

\begin{defn}\label{2.3.} Let $L$ be a multiplicative lattice and let $i \in L$. We associate an undirected and simple graph, called the {\it
zero-divisor graph of $L$ with respect to an element i}, denoted
by $\Gamma^m_i(L)$ in which the set of vertices is  $\{x(\nleq
i)\in L \;|\; x\cdot y \leq i$ for some $y(\nleq i)\in L\}$ and
two distinct vertices $a,b$ are adjacent if and only if $a \cdot b
\leq i$. Whenever $i=0$, we denote $\Gamma^m_i(L)$ by simply
$\Gamma^m(L)$. In this case, the vertex set of $\Gamma^{m}(L)$ is
called {\it the set of non-zero zero divisors of $L$} and denoted
by $(Z(L))^{*}$. Further,  we denote by $Z(L)=(Z(L))^{*} \cup
\{0\}$.
\end{defn}

\begin{exa}
Consider the lattice $L$ shown in Figure 1$(a)$ with the trivial
multiplication $x \cdot y =0=y \cdot x$, for each $x \neq 1 \neq
y$ and $x \cdot 1 =x=1 \cdot x$ for every $x \in L$. Then it is
easy to see that $L$ is a multiplicative lattice. Further, it's
zero divisor graph $\Gamma^m(L)$ (in the multiplicative lattice
sense) is shown in Figure 1$(c)$. It is interesting to note that
if 1 is completely join-irreducible({\it i.e.} $1=\bigvee x_i
\Rightarrow 1=x_i$ for some $i$) then any lattice with this
trivial multiplication is a multiplicative lattice.
\end{exa}

From the above example, it is clear that $\Gamma(L)$ and $\Gamma^m(L)$ need not be isomorphic. Hence it is natural to ask the following  question.

\noindent{\bf Question 2:} {\it Find a class of multiplicative lattices $L$ for which $\Gamma(L)\cong \Gamma^m(L)$.}

We answer this question in the following result.

\begin{lem}\label{equi} A multiplicative lattice  $L$ is  reduced if and only if
$\Gamma(L)=\Gamma^m(L)$.
\end{lem}

\begin{proof} Let $L$ be a reduced multiplicative lattice, then $\Gamma(L)=\Gamma^m(L)$ follows from the fact that $a.b=0$ if and only
if $a\wedge b=0$ for all $a, b \in L$. Conversely, Suppose that
$\Gamma(L)=\Gamma^m(L)$ and $L$ is not reduced lattice. Then there
exists an element $a (\neq 0) \in L$ such that $a^{n}=0$ with
$a^{n-1} \neq 0$ for some positive integer $n$. Let $b=a^{n-1}$,
then $a \cdot b =0$ gives $a$ and $ b$ is adjacent in
$\Gamma^m(L)$. This gives $a$ and $b$ are adjacent in $\Gamma(L)$.
But then we have $a^{n-1}=b=a \wedge b =0$, a contradiction.
\end{proof}

\begin{lem}\label{a}
Let $L$ be a reduced multiplicative lattice. Then $L$ is a 0-distributive lattice.
\end{lem}

\begin{proof}
Let $a \wedge b=0=a \wedge c$ for $a,b,c \in L$. Then $a.b=0=a.c$.
By the definition of a multiplicative lattice, we have $a. (b\vee
c)=0$. Since $L$ is reduced, this further gives $a \wedge (b \vee
c)=0$. Thus $L$ is 0-distributive.
\end{proof}

The following result follows from the fact that $L$ is commutative
semigroup and the result follows from DeMeyer et. al. \cite{DMS}.
But for the sake of completeness we provide the proof of it.

\begin{thm}\label{2.9.} Let $L$ be a multiplicative lattice, then $\Gamma^{m}(L)$
is connected with diameter $\leq 3$. Furthermore, if
$\Gamma^{m}(L)$ contains a cycle then $gr(\Gamma^{m}(L))\leq
4$.\end{thm}

\begin{proof} Let $x, y$ be distinct vertices in $\Gamma^{m}(L)$. Therefore there exists
$z\neq 0$, $w \neq 0$ with $ x \cdot z = y \cdot w = 0 $. If $x
\cdot y = 0$, then $x - y $ is a path. Now, suppose $x \cdot y \neq 0$. If $w \cdot z =0 $, then $x, y $ are
connected by a path $ x - z - w - y $ of length $\leq 3$. If $w \cdot z \neq 0 $, then $x, y $ are
connected by a path $ x - w \cdot z - y $ of length $= 2$. Thus
$\Gamma^{m}(L)$ is connected and $diam(\Gamma^{m}(L))\leq 3$.
Suppose that $girth(\Gamma^{m}(L)) \neq \infty$. Therefore, there
exists a cycle of minimal length $n$ in $\Gamma^{m}(L)$, say
$x_{1} \rightarrow x_{2} \rightarrow x_{3} \rightarrow \ldots
x_{n} \rightarrow x_{1}$. Let $n \geq 5$. The minimality of $n$
ensures that $x_{2} \cdot x_{4} \neq 0$. Let $(z \neq 0) \leq
x_{2} \cdot x_{4}$. Then, we have $x_{1} \cdot z \leq x_{1} \cdot
x_{2} =0$ and $x_{5} \cdot z \leq x_{5} \cdot x_{4} =0$. It
follows that $x_{1} \rightarrow z \rightarrow x_{5} \rightarrow
\ldots x_{n} \rightarrow x_{1}$ is a cycle of length $n-2$ in
$\Gamma(L)$. This contradicts the minimality of $n$. Therefore, we
have $n=3$ or $4$, which implies that $gr(\Gamma^{m}(L)) =3 ~ or ~
4$. Hence, in all cases we have $gr(\Gamma^{m}(L)) \leq
4$.\end{proof}

\begin{cor}[{Behboodi and Rakeei \cite[Theorem 2.1]{br1}}] \label{2.10.} For every ring $R$, the annihilating-ideal graph $\mathbb{AG}(R)$ is connected and $diam(\mathbb{AG}(R)) \leq 3$.
Moreover, if $\mathbb{AG}(R)$ contains a cycle, then $gr(\mathbb{AG}(R))\leq
4$.\end{cor}

The following theorem is proved by  Alizadeh, Maimani, Pournaki
and Yassemi \cite{ali2} and Joshi \cite{j} for posets. We quote
this result when the poset is a lattice.

\begin{thm}[{\cite[Theorem 2.14 ]{j} \cite[Theorem 3.2]{ali2}}] \label{2.4.}Let $L$ be a lattice with 0. Then the following statements are equivalent.

 \begin{enumerate} \item  There exist non-zero minimal prime
semi-ideals $P_{1}$ and $P_{2}$ of $L$ such that $P_{1} \cap
P_{2}={0}$. \item The zero-divisor graph $\Gamma(L)$ is a complete
bipartite graph. \item  The zero-divisor graph $\Gamma(L)$ is a
bipartite graph.\end{enumerate}\end{thm}

\begin{lem}\label{2.4a.}Let $L$ be a reduced multiplicative
lattice. If $p_{1}$ is minimal prime element in $L$, then
$(p_{1}]$ is prime ideal of $L$.\end{lem}

\begin{proof}  Suppose, $a \wedge b
\in (p_{1}]$, then $a \cdot b \leq a \wedge b \leq p_{1}$. As
$p_{1}$ is prime, $a \leq p_{1}$ or $b \leq p_{1}$ this gives $a
\in (p_{1}]$ or $b \in (p_{1}]$. Hence $(p_{1}]$ is prime
ideal.\end{proof}

\begin{thm}\label{2.5.} Let $L$ be a reduced multiplicative lattice. Then $\Gamma^m(L)$ is a complete
bipartite graph if and only if there exist minimal prime elements
$p_{1}$ and $p_{2}$ such that $p_{1} \wedge p_{2} = 0$.\end{thm}

\begin{proof} Suppose that $\Gamma^m(L)$ is a complete bipartite graph with
parts $V_{1}$ and $V_{2}$. Set $p_{1} = \bigvee V_{1}=
\bigvee\{x~|~ x \in V_1~\}$ and $p_{2} = \bigvee V_{2}$. For any $x_{p_2}\in V_2$, we have $x_{p_2} \cdot x_{p_1}=0$ for every $x_{p_1}\in V_1$. By the definition of a multiplicative lattice, $x_{p_2}\cdot (\bigvee_{x_{p_1}\in V_1} x_{p_1})=0$. This further yields $p_1\cdot p_2=0$. Since  $L$ is  reduced, $p_{1} \wedge p_{2} = 0$.

   Now, we prove $p_{1}$ is a prime element of $L$. Suppose, $a \cdot b \leq p_{1}$
and $a, b \nleq p_{1}$. Then $a, b \not\in V_{1}$ and $a, b \neq
0$. We claim that $a \cdot b \cdot c = 0$ for $c \in V_{2}$. As, $
a \cdot b \leq p_{1} $ and $c \leq p_{2}$ gives $ a \cdot b \cdot
c \leq  p_{1} \wedge p_{2}$. Therefore, $a \cdot b \cdot c = 0$
for any $c \in V_{2}$.

   Also, $b \cdot c \neq 0$. If possible, $b \cdot c =0$ then $ b\in
V_{1}$ (as $c \in V_{2} $), a contradiction to $b \not\in V_{1}$.
Thus $b \cdot c \neq 0$. This together with $a \cdot b \cdot c =
0$, we  get $b\cdot c \in V_{1}$ or $b\cdot c \in V_{2}$. If
$b\cdot c \in V_{2}$ then $a \in V_{1}$, again a contradiction.
Therefore $b\cdot c \in V_{1}$. Since $c\in V_{2}$, we have
$b\cdot c^{2}=0$ which further yields that $(b\cdot c)^{2}=0$.
Since $L$ is reduced, we have $b\cdot c=0$, again a contradiction. Thus
$p_{1}$ is a prime element of $L$. Similarly, we can show $p_{2}$
is a prime element of $L$. It is clear that every prime element contains a minimal prime element, we may assume that $p_{1}$ and $p_{2}$
are minimal prime elements with $p_1 \wedge p_2=0$.

    Conversely, suppose that there exist minimal prime elements
$p_{1}$ and $p_{2}$ such that $p_{1} \wedge p_{2} = 0$. By Lemma
\ref{2.4a.}, $(p_{1}]$ and $(p_{2}]$ are prime ideals of $L$. Clearly, $(p_{1}] \cap (p_{2}]=0$.
Since $L$ has zero, there exist minimal prime ideals $Q \subseteq
(p_{1}]$ and $R \subseteq (p_{2}]$ such that $Q \cap R =0$. By
Theorem \ref{2.4.}, the zero-divisor graph $\Gamma(L)$ is a
complete bipartite graph and hence the zero-divisor graph
$\Gamma^{m}(L)$ is a complete bipartite graph, by Lemma \ref{equi}.
\end{proof}

\begin{thm}\label{2.7.} Let $L$ be a reduced multiplicative lattice. Then following statements are equivalent:
\begin{enumerate} \item  There exist non-zero minimal prime
ideals $P_{1}$ and $P_{2}$ of $L$ such that $P_{1} \cap
P_{2}={0}$. \item  The zero-divisor graph $\Gamma(L)$ is a
complete bipartite graph.
\item  The zero-divisor graph $\Gamma(L)$ is a
 bipartite graph.
\item  The zero-divisor graph
$\Gamma^{m}(L)$ is a complete bipartite graph. \item There exist
nonzero minimal prime elements $p_{1}$ and $p_{2}$ of $L$ such
that $p_{1} \wedge p_{2}={0}$.\end{enumerate}\end{thm}

\begin{proof} follows from the Theorem \ref{2.4.} and
Theorem \ref{2.5.}.\end{proof}

\begin{rem}\label{b1} It is  known that if $R$ is a reduced (non-reduced)
commutative ring with unity. Then $Id(R)$, the ideal lattice of
$R$, is reduced (non-reduced) multiplicative lattice which is
$1$-compact and compactly generated. Further if $R$ is a reduced
commutative ring with unity and $L=Id(R)$, then $P$ is a minimal prime ideal of $R$ if and only if $P$
is a minimal prime element of $L$.
\end{rem}

From Theorem \ref{2.7.} and  Remark \ref{b1}, we have the  following result.

\begin{cor}[{Behboodi and Rakeei \cite[Corollary 2.5]{br2}}]\label{2.8.} Let $R$ be a reduced ring. Then the
following statements are equivalent:
\begin{enumerate}
 \item $\mathbb{AG}(R)$ is a complete bipartite graph
with two non-empty parts. \item $R$ has exactly two minimal
primes.\end{enumerate}\end{cor}

\begin{lemma}[{Alarcon, Anderson, Jayaram \cite[Lemma 3.5]{AAJ}}]\label{mini} Let $L$ be a $1$-compact, compactly generated lattice with $L_{*}$as a  multiplicatively closed  set and let $p$ be a prime element of $L$. Then $p$ is a minimal prime element of $L$ if and only if for any compact element $x \leq p$ there exists a compact element $y \not\leq p$ such that $x^n \cdot y=0$ for some positive integer $n$.
\end{lemma}

\begin{thm}\label{2.6.} Let $L$ be a compactly generated, reduced multiplicative lattice having more than two minimal prime elements, then  $\Gamma^m(L)$ contains a cycle and $gr(\Gamma^{m}(L)) = 3$.\end{thm}

\begin{proof} Let $p_1, p_2$ and $p_3$ be three distinct minimal prime elements of $L$. Since $L$ is compactly generated and $p_1\not\leq p_2, p_3$, there exist $x_1, x_1'\in L_*$ such that $x_1, x_1' \leq p_1$, $x_1' \not\leq p_2$ and $x_1\not\leq p_3$. Hence there is $x_1\vee x_1'=z_1\in L_*$ such that $z_1 \leq p_1$ with $z_1 \not\leq p_2, ~p_3$.  We claim that $z_1 \cdot p_2 \not=0$, otherwise, $z_1 \cdot p_2 \leq p_3$, a contradiction. Since $z_1 \leq p_1$, by Lemma \ref{mini}, there exists $y_1\in L_*$ such that $y_1\not\leq p_1$ with $z_1 \cdot y_1=0$. Now, choose $(\not=0) t\in L_*$ such that $t \leq  z_1 \cdot p_2 \leq p_2$.  By Lemma \ref{mini}, there exists $y_2\in L_*$ such that $y_2 \not\leq p_2$ with $t \cdot  y_2=0$.

If $y_1 \cdot y_2=0$. Then $a,b,c$ forms a cycle, where $a=t$, $b=y_1$ and $c=y_2$.

We assume that   $y_1 \cdot y_2\not=0$.  Since $p_2$ is prime and $z_1 \cdot y_1 =0 $, we get $y_1\cdot y_2 \leq p_2$.  Since $L_*$ is multiplicatively closed, by Lemma \ref{mini}, there exists $z_2 \not\leq p_2$ such that $ y_1\cdot y_2 \cdot z_2 =0$.
Put $a=t$, $b=y_1\cdot y_2$ and  $c=y_2 \cdot z_2$. Clearly, $a,b,c \in V(\Gamma^m(L))$ with $a \cdot b=b\cdot c= a\cdot c=0$. Thus $\Gamma^m(L)$ contains a cycle.

Now, we prove that $gr(\Gamma^m(L))=3$. On the contrary, suppose $gr(\Gamma^{m}(L)) = 4$.  Then $gr(\Gamma^{m}(L)) = gr(\Gamma(L)) =
4$, by Lemma \ref{equi}. As $\Gamma(L)$ contains a cycle,
$\Gamma(L)$ is not a star graph. We show that $\Gamma(L)$ has no
odd cycle. Then, in the view of the well known result of K\"onig, $\Gamma(L)$ is bipartite. On the contrary, assume that $\Gamma(L)$ has an odd cycle and let $x_{1}
\rightarrow x_{2} \rightarrow x_{3} \rightarrow \ldots x_{n}
\rightarrow x_{1}$ be an odd cycle of minimal length $n$ in
$\Gamma(L)$. Clearly $n \geq 5$, since $gr(\Gamma(L)) \neq 3$.
Now, the minimality of $n$ ensures that $x_{2} \cdot x_{4} \neq
0$. Let $(0 \neq) z  \leq x_{2} \cdot x_{4} $. Then we have $x_{1}
\cdot z \leq x_{1} \cdot x_{2} =0$ and $x_{5} \cdot z \leq x_{5}
\cdot x_{4} =0$. It follows that $x_{1} \rightarrow z \rightarrow
x_{5} \rightarrow \ldots x_{n} \rightarrow x_{1}$ is an odd cycle
of length $n-2$ in $\Gamma(L)$. This contradicts the minimality of
$n$. Hence, $\Gamma(L)$ has no odd cycle. Therefore $\Gamma(L)$ is bipartite.
By Theorem \ref{2.4.}, $\Gamma(L)$ is a complete bipartite graph. Hence
$\Gamma^{m}(L)$ is complete bipartite. By Theorem \ref{2.5.},
there exist minimal prime elements $p_{1}$ and $p_{2}$ in $L$ such
that $p_{1} \wedge p_{2} = 0$.  We claim that there are the only two minimal prime elements in $L$.
Suppose, there exists a third
minimal prime element, say $p_{3} \not\in \{ p_{1}, p_{2}\} $ in $L$.
As $p_1\cdot p_2 =p_{1} \wedge p_{2} = 0 \leq p_{3}$, gives $p_{1} \leq p_{3}$
or $ p_{2}\leq p_{3}$, a contradiction to the minimality of $p_{3}$.
Thus $L$ has exactly two minimal prime elements, a
contradiction to the assumption that $L$ has more than two minimal
prime elements. Hence $gr(\Gamma^{m}(L)) = 3$.\end{proof}

Now, Remark \ref{b1} and Theorem \ref{2.6.}  settles Conjecture $1.11$ of Behboodi and Rakeei \cite{br2}.

\begin{cor}[{Behboodi and Rakeei \cite[Conjecture 1.11]{br2}}]\label{2.6a.} Let $R$  be a reduced ring with
more than two minimal primes. Then $gr(\mathbb{AG}(R))= 3$.\end{cor}

\begin{lemma} \label{2.11.} Let $L$ be a multiplicative lattice. If $Z(L)$ is an ideal, then $diam(\Gamma^m(L)) \leq 2$.\end{lemma}

\begin{proof} Let $x, y \in V(\Gamma^m(L))$. If $x \cdot y = 0$, then $d(x,y)=1$. Suppose $d(x,y)\neq 1$.
Since $Z(L)$ is an ideal, $x \vee y \in V(\Gamma^{m}(L))$ and
hence $(0: (x \vee y)) \neq 0$. We have  $(x \vee y) \cdot (0: (x
\vee y)) = 0$. As $x \cdot (0: (x \vee y))
 =0$ and $y \cdot (0: (x \vee y)) =0$, $d(x, y) = 2$. Hence,
$diam( \Gamma^m(L)) \leq 2$. \end{proof}

\begin{lem}\label{2.18.} Let $L$ be a non-reduced multiplicative lattice. If $a, b \in V(\Gamma^{m}(L))$
and $q \in Nil(L)$, then $a \vee (b \cdot q) \in
V(\Gamma^{m}(L))$.\end{lem}

\begin{proof} Let $q$ be a non-zero nilpotent element. Since $a \in
V(\Gamma^{m}(L))$, there exist $c \neq 0$ such that $c \cdot a =
0$. Since $q$ is nilpotent, there is a positive integer $m$ such
that $c \cdot q^{m} = 0$ and $c \cdot q^{m-1} \neq 0$. Consider
the pair $a \vee (b \cdot q) \neq 0$ and $c \cdot q^{m-1} \neq 0$,
then $(a \vee (b \cdot q)) \cdot (c \cdot q^{m-1}) = 0$. Therefore
$a \vee (b \cdot q) \in V(\Gamma^{m}(L))$.\end{proof}

\begin{thm}\label{2.20.} Let $L$ be non-reduced multiplicative lattice and $Z(L)$ is not an ideal, then $diam(\Gamma^{m}(L)) = 3$.\end{thm}

\begin{proof} Since $Z(L)$ is not an ideal, there exist $a, b \in
 V(\Gamma^m(L))$ such that $0=(0:(a \vee b))$.

$Case(I):$ If $a$ and $b$ are non-adjacent and if $d(a, b) = 2$, then there is an
element $c \neq 0$ such that $a \cdot c, b \cdot c = 0$. This
gives $c \cdot (a \vee b))= 0$, a contradiction to the fact that $0=(0:(a \vee
b))$. Therefore $d(a, b) = 3$.

$Case(II):$ Now, suppose $a$ and $b$ are adjacent, that is, $a \cdot b = 0$.
Then $(a \vee b)^{2} = a^{2} \vee b^{2}$. We
claim that $(0:(a^{2} \vee b^{2}))=0$. If possible, $(0:(a^{2}
\vee b^{2})) \neq 0$, then there exist $c \neq 0$ such that $c
\cdot (a^{2} \vee b^{2})=0$, that is, $c \cdot (a \vee b)^{2}=0$. But then $0=(0:(a \vee b))$ gives $c\cdot (a \vee b)=0$ which further yields $c=0$, a contradiction.

Now, we claim that there is a non-zero nilpotent element $q$ such that either $a^2 \cdot q\not=0$ or $b^2 \cdot q \not=0$. Suppose on  the contrary that $a^2 \cdot q=0$ and $b^2 \cdot q =0$. This together with $a\cdot b \cdot q=0$ gives $(a\vee b) \cdot b\cdot q =0$. Since $0=(0:(a \vee b))$, we have $b \cdot q=0$. Similarly, we have $a\cdot q=0$. Hence $(a\vee b)\cdot q =0$. Thus $q=0$, a contradiction.

Hence without loss of generality, we may assume that there is a non-zero nilpotent
element $q$ such that $b^{2} \cdot q \neq 0$. Since $a, b \in
V(\Gamma^{m}(L))$ and $q \in Nil(L)$, by Lemma \ref{2.18.}, $a
\vee (b \cdot q) \in V(\Gamma^{m}(L))$. Consider the pair $a \vee
(b \cdot q)$ and $b$. Clearly, $(a \vee (b \cdot q)) \cdot b =
b^{2} \cdot q \neq 0$. Hence $d(a \vee (b \cdot q), ~b) \neq 1$.
Since $(0:(a \vee b)) = (0:(b \vee a \vee (b \cdot q)))=0$, applying technique as in Case$(I)$ for $a \vee (b\cdot q)$ and $b$ we get $d(a
\vee (b \cdot q), ~b) \neq 2$.  Hence by Theorem \ref{2.9.}, $d(a
\vee (b \cdot q), ~b) = 3$ and thus  $diam(\Gamma^{m}(L)) =
3$.\end{proof}

\begin{thm}\label{2.22.} Let $L$ be non-reduced multiplicative lattice. Then $Z(L)$ is not an ideal if and only if $diam(\Gamma^{m}(L)) = 3$\end{thm}

\begin{proof} Follows from the Theorem \ref{2.11.} and Theorem
\ref{2.20.}.\end{proof}

\begin{cor}[{Behboodi and Rakeei \cite[Theorem 1.6]{br2}}]\label{2.23.} Let $R$ be a non-reduced ring. There is a pair of
annihilating-ideals $I$, $J$ of $R$ such that $Ann(I+J)=(0)$ if
and only if $diam(\mathbb{AG}(R))=3$.\end{cor}

\begin{lem}\label{2.11a.} Let $L$ be a reduced lattice and
$diam(\Gamma^m(L))=1$. Then $Z(L)$ is not an ideal.\end{lem}

\begin{proof} Let $x \neq 0 $ and $y \neq 0$ be distinct elements
of $Z(L)$. Suppose $Z(L)$ is an ideal. Therefore $x \vee y \in
Z(L)$. If $x \neq (x \vee y)$, then by $diam(\Gamma^m(L))=1$ gives $x \cdot (x \vee y) = x^{2}
\vee (x \cdot y) = x^{2}=0$. This gives $x=0$, a contradiction.
Hence $x = x \vee y$. Similarly, $y = x \vee y$.  Hence $x=y$,
a contradiction. Thus $Z(L)$ is not an ideal.\end{proof}

\begin{lem} Let $L$ be reduced lattice and $Z(L)$ is an ideal, then $diam(\Gamma^m(L))=2$. \end{lem}
\begin{proof} Proof follows from the Lemma \ref{2.11.} and Lemma \ref{2.11a.}.\end{proof}

The following theorem is due to  Joshi, Waphare and  Pourali
\cite{jwp}. We quote this result when $I=\{0\}$ is a semiprime ideal,  that is, $L$ is a $0$-distributive lattice.

\begin{thm}\label{2.13.} Let $L$ be a $0$-distributive lattice and $V(\Gamma(L))\cup  \{0\}$ is not an ideal. Then $diam(\Gamma(L)) \leq 2$ if and only if $L$ has exactly
two minimal prime ideals.\end{thm}

\begin{thm}\label{2.14.} Let $L$ be a  1-compact, compactly generated multiplicative lattice  with $L_*$ as a multiplicatively closed set and $Z(L)$ is not an ideal. Then $diam(\Gamma^m(L)) = 2$ if and only if $L$ is reduce with exactly two
minimal prime elements.\end{thm}

\begin{proof} Suppose that $diam(\Gamma^m(L)) = 2$. By Theorem \ref{2.20.}, $L$ is
a reduced lattice. Hence  by Lemma \ref{equi}, Lemma \ref{a} and Theorem \ref{2.13.},
$L$ has exactly two minimal prime ideals, say $P$ and $Q$. Since $L$ is a 0-distributive lattice, intersection of all minimal prime ideals of $L$ is zero . Hence $P \cap Q ={\{0\}}$. Therefore by Theorem \ref{2.7.}, $L$ has two minimal prime elements say $p_{1}$ and $p_{2}$ such that $p_{1}
\wedge p_{2}=0$. Now, we show $p_1$ and $p_2$ are the only two minimal
prime elements $L$. Suppose $L$ has a third minimal prime element, say
$p_{3}$. Then $p_{1} \wedge p_{2}=0 \leq p_{3}$ gives $p_{1} \leq
p_{3}$ or $p_{2} \leq p_{3}$, a contradiction to the minimality of
$p_{3}$. Hence $L$ has exactly two minimal prime elements $p_1,p_2$.

         Conversely, suppose that $L$ is reduce with exactly two minimal prime
elements, say $p_1, p_2$. By Corollary \ref{1.4.}, we have $p_1 \wedge p_2=0$. Therefore by Theorem \ref{2.7.}, $L$ has two minimal
prime ideals, say $P$ and $Q$ such that $P \cap Q =\{0\}$.  Suppose $L$
has third a minimal prime ideal, say $R$. Then $P \cap Q =\{0\} \subseteq
R$ gives $P \subseteq R$ or $Q \subseteq R$, a contradiction to the
minimality of $R$. Hence, $L$ has exactly two minimal prime
ideals. By Theorem \ref{2.13.}, $diam(\Gamma(L)) \leq 2$ and hence
$diam(\Gamma^{m}(L)) \leq 2$. Thus $diam(\Gamma^m(L))=2$, by Lemma \ref{2.11a.}.
\end{proof}

\begin{cor}\label{2.15.} Let $R$ be a reduced ring such that $Z(R)$ is not an ideal. Then  $diam(\mathbb{AG}(R)) = 2$ if and only if $R$ has exactly two
minimal primes.\end{cor}

\begin{thm}\label{2.24.} Let $L$ be a $1$-compact, compactly generated lattice with $L_{*}$ as a multiplicatively closed set and $Z(L)$ is not an
ideal. Then $diam(\Gamma^m(L)) = 3$ if and only if either $L$ is
reduced with more than two minimal prime elements or $L$ is
non-reduced.\end{thm}

\begin{proof} (1) In the case, $L$ is a reduced lattice with more than two minimal prime
elements, then $diam(\Gamma^m(L)) = 3$ follows from Theorem \ref{2.14.} and in the case $L$ is non-reduced,
then the result follows from Theorem \ref{2.22.}.

      Conversely, suppose that $diam(\Gamma^m(L)) = 3$.

      (1) Suppose $L$
is a reduced lattice. If $L$ has exactly one minimal prime
element, then  $V(\Gamma^m(L)) = \emptyset$, a contradiction.
Therefore $L$ can't have exactly one minimal prime element. Now, if
$L$ has exactly two minimal prime element, then by Theorem
\ref{2.14.}, $diam(\Gamma^m(L))= 2$, again a contradiction.
Therefore $L$ has more than two minimal prime elements.

(2) By Theorem \ref{2.14.}, $L$ is non-reduced.\end{proof}

\begin{cor}[{Behboodi and Rakeei \cite[Theorem 1.3]{br2}}] \label{2.25.} Let $R$ be a ring such that $Z(R)$ is
not an ideal. Then the following statements are equivalent:

\begin{enumerate} \item $diam(\mathbb{AG}(R))=3$ \item Either $R$
is non-reduced or $R$ is a reduced ring with more than two minimal
primes.\end{enumerate}\end{cor}

\begin{thm}\label{2.26.} Let $L$ be  $1$-compact, compactly generated lattice with $L_{*}$ as a multiplicatively closed set and $Z(L)$ is not an
ideal. Then $1 \leq diam(\Gamma^m(L)) \leq 3$ and
\begin{enumerate} \item $diam(\Gamma^m(L))=1$ if and only if $L$
is reduced and $|(Z(L))^{*}|=|A(L)|=2$, where $A(L)$ is the set of
atoms of $L$; \item $diam(\Gamma^m(L)) = 2$ if and only if $L$ is
reduced with exactly two minimal prime elements and $|(Z(L))^{*}|
> 2$; \item $diam(\Gamma^m(L)) = 3$ if and only if $L$ is reduced
with more than two minimal prime elements or $L$ is
non-reduced.\end{enumerate}\end{thm}

\begin{proof}  By Theorem \ref{2.9.}, $diam(\Gamma^m(L)) \leq 3$. If $diam(\Gamma^m(L))=0$, then $\Gamma^m(L)$ has only one vertex,
that is, $Z(L)$ is an ideal of $L$, a contradiction to our
assumption that $Z(L)$ is not an ideal of $L$. Hence $1 \leq
diam(\Gamma^m(L)) \leq 3$.

(1)  Suppose $diam(\Gamma^m(L))=1$. By Theorem
\ref{2.22.}, $L$ is a reduced lattice. As $L$ is reduced lattice,
$A(L) \subseteq (Z(L))^{*}$. Let $x \in (Z(L))^{*}$. If $x$ is an
atom, then we are through. If $x$ is not an atom, then there
exists $y$ such that $0 < y< x$. Clearly, $y \in (Z(L))^{*}$, as $x \in
(Z(L))^{*}$. Since $diam(\Gamma^m(L))=1$ and $L$ is a reduced,
 we must have $x\cdot y=x \wedge y =0$. Therefore $y =0$, a
contradiction to $y \in (Z(L))^{*}$. Hence $x$ is an atom and
$(Z(L))^{*} = A(L)$. Now assume that $|(Z(L))^{*}|=|A(L)| > 2$.
Let $p, q, r$ be any three distinct atoms of $L$. Clearly, $p \cdot q = p
\cdot r = 0$ and this gives $p \cdot (q \vee r)=0$. This implies
$(q \vee r) \in (Z(L))^{*} = A(L)$, a contradiction. The converse
is obvious.

(2) Follows from  Theorem \ref{2.14.}.

(3) Follows from Theorem \ref{2.24.}.
\end{proof}

\begin{lem}\label{2.29.} Let $L$ be a reduced, $1$-compact, compactly generated lattice with $L_{*}$ as a multiplicatively closed set
having finitely many minimal prime elements. Then $ x \in Z(L)$ if
and only if $x$ is contained in at least one minimal prime
element.\end{lem}

\begin{proof}Let $ x \in Z(L)$. If $x =0$, then we are through. If $x \neq
0$, then there exists  $y \neq 0$ such that $x \cdot y = 0$. Since
$L$ has finitely may minimal prime elements, say $p_{1},
p_{2},....p_{n}$ and $L$ is reduced lattice, by Corollary \ref{1.4.},   $0 =
p_{1} \wedge p_{2} \wedge .... \wedge p_{n}$. If $x$ is not
contained in any minimal prime element, then $y \leq p_{1} \wedge
p_{2} \wedge .... \wedge p_{n}=0$, a contradiction to $y \neq 0$.
Hence $x$ is contained in at least one minimal prime element.

Now, we show that every minimal prime element is in $Z(L)$. Without
loss of generality we will show that $p_{1}$ is in $Z(L)$. We
claim that $p_{i} \wedge p_{i+1} \wedge .... \wedge p_{n} \neq 0$
for $i \geq 2$. If not, then  $p_{i} \wedge p_{i+1}
\wedge .... \wedge p_{n} = 0 \leq p_{1}$ for $i \geq 2$. This
gives some $p_{j} \leq p_{1}$ for $2 \leq j \leq n$, contradicting
the minimality of $p_{1}$. Hence $p_{i} \wedge p_{i+1} \wedge ....
\wedge p_{n} \neq 0$ for $i \geq 2$, and this together with $p_1 \wedge (p_2\wedge p_3 \wedge \cdots p_n)=0$ gives $p_{1}\in Z(L)$. Hence every minimal prime element is in $Z(L)$. Therefore
if $x$ is contained in at least one minimal prime element then $ x
\in Z(L)$. \end{proof}

\begin{lem}\label{2.29a.} Let $L$ be a reduced, $1$-compact, compactly generated lattice with $L_{*}$ as a multiplicatively closed set
having finitely many minimal prime elements. If $p$ is minimal
prime element of $L$, then $p$ contain precisely one of $a$ and
$a^{*}$.\end{lem}

\begin{proof} Let $a \leq p$.  By Lemma \ref{b1}, there
exist $b \not\leq p$ such that $a \cdot b =0$. Therefore $a^{*}
\nleq p$. Let $a^{*} \leq p$. Clearly, we  have  $a \not\leq p$, otherwise $a^* \not\leq p$, a contradiction.\end{proof}

\begin{lem}\label{2.30.} Let $L$ be a reduced, $1$-compact, compactly generated lattice with $L_{*}$ as a  multiplicatively closed set
having finitely many minimal prime elements. If $L$ has more than
two minimal prime elements, then $Z(L)$ is not an ideal and hence
$diam(\Gamma^{m}(L))=3$.\end{lem}

\begin{proof} Let $L$ have more than
two minimal prime elements. Let $a , a^{*} \in Z(L)$. We claim
that $a \vee a^{*} \not\in Z(L)$. If $a \vee a^{*} \in Z(L)$, then
by Lemma \ref{2.29.} $a \vee a^{*}$ is contained in at least one
minimal prime element, say $q$ of $L$. But then by Lemma \ref{2.29a.}, the minimal prime element $q$ contains precisely one
of $a$ and $a^{*}$. Hence $Z(L)$ is not an ideal. Thus by Theorem
\ref{2.24.}, $diam(\Gamma^{m}(L))=3$.
\end{proof}

\begin{cor}[{Behboodi and Rakeei \cite[Lemma 1.8]{br2}}]\label{2.31.} Let $R$ be a reduced ring with finite
minimal primes. If $R$ has more than two minimal primes, then
$diam(\mathbb{AG}(R))=3$.\end{cor}

\end{document}